\documentclass[12pt,reqno,oneside]{amsart}
\usepackage{amsmath,amsthm}

\usepackage{enumerate}

\usepackage{macro}

\newtheorem{thm}{Theorem}[section]
\newtheorem{prop}[thm]{Proposition}
\newtheorem{defn}[thm]{Definition}
\newtheorem{lem}[thm]{Lemma}

\newcommand{\Hm}{\mathcal{H}}
\newcommand{\abb}[1]{{#1}_{\alpha\bar{\beta}}}
\newcommand{\uabb}[1]{{#1}^{\alpha\bar{\beta}}}

\newcommand{\ua}[1]{{#1}^{\alpha}}
\newcommand{\ubb}[1]{\overline{{#1}^{\beta}}}

\newcommand{\bfb}{\bar{\beta}}
\newcommand{\bfa}{\bar{\alpha}}
\newcommand{\bfg}{\overline{\gamma}}

\newcommand{\lapx}{\tilde{\lap}}

\renewcommand{\Re}{\text{Re}}

\DeclareMathOperator{\osc}{osc}

\begin{document}


\title{Complex Hessian Equation on Kahler Manifold}
\author{Zuoliang Hou}
\address{Mathematics Department, Columbia University, New York, NY 10027}
\email{hou@math.columbia.edu}
\date{\today}

\maketitle

\section{Introduction}

Let $(M,\fo)$ be a compact \KH{} manifold of complex dimension
$n$. In 1978, Yau~\cite{Yau1978} proved the famous Calabi-Yau
conjecture by solving following complex \MA{} equation on $M$
$$  (\fo+\ii\ddb\vff)^n = f \fo^n $$
with positive function $f$. Later on, using tools from pluri-potential
theory, Ko{\l}odziej~\cite{Kolodziej1998} studied the same equation with
weaker smoothness assumption on $f$.

\medskip
In this paper, we consider following complex Hessian equation on
$(M,\fo)$: 
\begin{equation}
  \fo_{\vff}^k \wedge \fo^{n-k} = 
  (\fo+\ii\ddb\vff)^k \wedge \fo^{n-k} = f \fo^n, \label{eq:1}
\end{equation}
where $k$ is a fixed integer between $2$ and $n-1$, and $f$ is a
non-negative function on $M$ satisfying the compatibility condition:
\begin{equation}
  \int_M f \fo^n = \int_M \fo^n.  \label{eq:2}
\end{equation}
Noticed that if $k=n$, equation (\ref{eq:1}) is just the complex \MA{}
equation, while if $k=1$, equation (\ref{eq:1}) becomes the Laplacian
equation. So equation~(\ref{eq:1}) is a generalization of both complex
\MA{} equation and Laplacian equation over compact \KH{} manifold.

\medskip
Similar nonlinear equations have been studied extensively by many
authors, see \cite{BedfordTaylor1976, CNS1983, CKNS1985, CNS1985,
Li1990, Trudinger1995, GuanLi1996, TrudingerWang1999, Li2004,
Blocki2005} and the reference there.

\bigskip
The Main result of this paper is following 
\begin{thm} Let $(M,\fo)$ be a compact \KH{} manifold with
  non-negative holomorphic bisectional curvature, and $f$ is a strictly
  positive smooth function, then equation (\ref{eq:1}) has smooth solution
  unique up to a constant. 
  \label{thm:1.1}
\end{thm}

Our approach to Theorem~\ref{thm:1.1} is similar to Yau's approach to
the complex \MA{} equation, i.e. continuity method and \emph{a priori}
estimate. By the standard theory of Krylov and Evans, it suffices to
prove \emph{a priori} $C^2$ estimate for equation (\ref{eq:1}). More
precisely, we have 

\begin{prop} If $\vff$ solves equation (\ref{eq:1}) and $\sup_M \vff =0$,
  then $\forall\,q > 2n$, 
  \begin{equation}
	\norm{\vff}_{C^0} \leq C_q, \label{eq:3}
  \end{equation}
  where $C_q$ is a constant depends on $(M,\fo)$, $q$ and
  $\norm{f-1}_{L^q}$.
  \label{prop:1.2}
\end{prop}

\begin{prop} If $(M,\fo)$ has non-negative holomorphic bisectional
  curvature and  $\vff$ solves equation (\ref{eq:1}), then 
  \begin{equation}
	\norm{\grad \vff}_{C^0(\fo)} \leq C_1, \label{eq:4}
  \end{equation}
  where $C_1$ is a constant depends on $(M,\fo)$, 
  $\norm{f^{1/k}}_{C^1(\fo)}$ and $\osc \vff$.
  \label{prop:1.3}
\end{prop}

\begin{prop} If $(M,\fo)$ has non-negative holomorphic bisectional
  curvature and  $\vff$ solves equation (\ref{eq:1}), then 
  \begin{equation}
	\norm{\ddb \vff}_{C^0(\fo)} \leq C_2, \label{eq:5}
  \end{equation}
  where $C_2$ is a constant depends on $(M,\fo)$, $\sup f$,  $\inf
  \lap_{\fo}(f^{1/k})$ and $\osc \vff$.
  \label{prop:1.4}
\end{prop}

\medskip
Both Proposition~\ref{prop:1.3} and Proposition~\ref{prop:1.4} require
non-negative holomorphic bisectional curvature for the underline \KH{}
manifold. However Yau's result requires no curvature condition, even
though complex \MA{} equation is worse than complex Hessian equation
in certain sense, because \MA{} equation is more nonlinear.
One possible explanation is that the convex cone of positive
real $(1,1)$ form is independent of the \KH{} metric, while the
convex cone of $k$-positive real $(1,1)$ form does depend on the
\KH{} metric. Besides, $\fo_{\vff}$ being positive is a much stronger
condition than $\fo_{\vff}$ being $k$-positive. In fact, Yau used the
positivity of $\fo_{\vff}$ to control some third order terms in his
proof of \emph{a priori} $C^2$ estimate. Also noticed that
Li~\cite{Li1990} studied some nonlinear equations with certain
structure conditions over compact Riemannian manifold which include
the real Hessian equation as a special case. In Li's treatment, the
non-negativity of sectional curvature is needed.  Right now,
We don't know whether the non-negativity of holomorphic bisectional
curvature is an essential condition for the solvability of equation
(\ref{eq:1}) or it is just a technical requirement.\footnote{In a
recent paper~\cite{HMW2008}, we removed the curvature assumption
imposed in this paper.}

\smallskip
We don't require the function $f$ to be strictly positive in both
Proposition~\ref{prop:1.3} and Proposition~\ref{prop:1.4}, i.e. 
equation~(\ref{eq:1}) can be degenerate. However, in order to use
the theory of Krylov and Evans to get higher regularity, we need the
equation to be uniformly elliptic. Hence in Theorem~\ref{thm:1.1}, we
require $f$ to be strictly positive.

\medskip B{\l}ocki~\cite{Blocki2005} studied the weak solution of
complex Hessian equation over bounded domain in $\CC^n$. The concept
of weak solution and complex Hessian measure can be extended to the
study of complex Hessian equation over compact \KH{} manifold, and one can
also study the corresponding potential theory. 

\medskip
We organize the rest of the paper as follows: in
Section~\ref{sec:preliminary}, we provide some necessary results on
convex cones related to elementary symmetric functions; in
Section~\ref{sec:uniqueness}, we study the uniqueness of solution
and the $C^0$ estimate; in Section~\ref{sec:C1} and
Section~\ref{sec:C2}, we derive the \emph{a priori} $C^1$ and
$C^2$ estimate respectively.

\section{Preliminary}
\label{sec:preliminary}

Let $S_k$ be the normalized $k$-th elementary symmetric
function
defined on $\RR^n$ and let
$$ \fG_k = \{ \fl \in \RR^n \,\mid\, S_j(\fl) > 0, j=1, \cdots, k\;\}.$$
It is well known that $\fG_k$ is an open convex cone in $\RR^n$. We
call $\fG_k$ the $k$-positive cone in $\RR^n$.

\begin{prop} For the $k$-positive cone $\fG_k$ in $\RR^n$, 
\begin{itemize}
  \item $\fG_n = \{ \fl = (\fl_1,\cdots, \fl_n) \in \RR^n \, \mid\,
	\fl_j >0 , j=1, \cdots, n \,\}.$ We call $\fG_n$ the positive cone
	in $\RR^n$.
  \item $\fG_k$ is the  connected component of
	$\{\fl\in\RR^n\,\mid\,S_k(\fl)>0\,\}$ containing $\fG_n$.
  \item $\forall\,\fl \in \fG_k, \forall\,j=1,\cdots,n$, 
    \begin{equation}
      \dod{S_k}{\fl_j}(\fl) > 0.
      \label{eq:6}
    \end{equation}
  \item (\emph{Newton inequalities}) $\forall\, \fl \in \RR^n,
	\forall\,j=1,\cdots,n-1$, 
    $$ S_{j-1}(\fl)S_{j+1}(\fl) \leq S_j(\fl)^2. $$
  \item (\emph{Maclaurin inequalities}) For $\fl \in \fG_k$, 
    \begin{equation}
	  0 < S_k^{1/k}(\fl) \leq \cdots \leq S_2^{1/2}(\fl) \leq S_1(\fl).
      \label{eq:7}
    \end{equation}
  \item (\emph{Generalized Newton-Maclaurin inequalities}) For $\fl \in
    \fG_k$
    $$ \Big(\frac{S_k(\fl)}{S_l(\fl)}\Big)^{\frac{1}{k-l}} \leq 
    \Big(\frac{S_r(\fl)}{S_s(\fl)}\Big)^{\frac{1}{r-s}}, $$ 
	where $0 \leq s < r$, $0 \leq l < k$, $ r \leq k$ and $s \leq l$. 
  \item (\emph{G{\aa}rding inequality}) Let $P_k$ be the complete
	polarization of $S_k$, then for $\fL^{(1)}, \cdots ,\fL^{(k)} \in \fG_k$,
    \begin{equation}
	  P_k(\fL^{(1)}, \cdots, \fL^{(k)}) \geq S_k(\fL^{(1)})^{1/k} \cdots
	  S_k(\fL^{(k)})^{1/k}. 
      \label{eq:8}
    \end{equation}
  \item $S_k^{1/k}$ is concave on $\fG_k$.
  \item If $0\leq l < k$, then $(S_k/S_l)^{1/(k-l)}$ is concave on $\fG_k$. 
\end{itemize}
\label{prop:2.1}
\end{prop}
The properties listed above are well known, for proof see
\cite{CNS1985,Garding1959}. 
In this paper,  we also need following result.
\begin{lem} Suppose $\fl=(\fl_1,\cdots,\fl_n) \in \fG_k$ with 
  \label{lem:2.1}
  $$ \fl_1 \geq \fl_2 \geq \cdots \geq \fl_n, $$
then
\begin{equation}
  \label{eq:9}
  \dod{S_k}{\fl_1} \leq \dod{S_k}{\fl_2} \leq \cdots \leq 
  \dod{S_k}{\fl_n}.
\end{equation}
Moreover
\begin{equation}
  \label{eq:10}
  \fl_1 \dod{S_k}{\fl_1} \geq \frac{k}{n} S_k.
\end{equation}
\end{lem}
\begin{proof} Equation (\ref{eq:9}) is obvious, we will only prove
  equation~(\ref{eq:10}). Let $\fs_k$ be the $k$-th elementary
  symmetric function, and $\fL=\{\fl_1,\cdots, \fl_n\}$, then
  $$ n \fl_1 \dod{\fs_k}{\fl_1} - k \fs_k = (n-k)\fl_1
  \fs_{k-1}(\fL\backslash \fl_1 ) - k \fs_k(\fL\backslash \fl_1). $$
  By Newton's Inequality for $\fL\backslash\fl_1$,
  $$ \fl_1 \geq S_1(\fL\backslash\fl_1) \geq
  \frac{S_k(\fL\backslash\fl_1)}{S_{k-1}(\fL\backslash\fl_1)} =
  \frac{k\fs_k(\fL\backslash\fl_1)}{(n-k)\fs_{k-1}(\fL\backslash\fl_1)}.$$
  So
  $$ n\fl_1 \dod{S_k}{\fl_1} \geq k S_k. $$
\end{proof}

\bigskip
Let $\Hm(n)$ be the set of $n\times n$ Hermitian matrices. We extend
the definition of $S_k$ on $\RR^n$ to $\Hm(n)$ by
$$ S_k(A) = S_k\big(\fl(A)\big)$$
where $\fl(A)$ is the eigenvalues of Hermitian matrix $A$, then $S_k$
is a homogeneous polynomial of degree $k$ on $\Hm(n)$ and
\begin{equation}
	\det(A+tI_n) = \sum_{j=0}^n \binom{n}{j} S_j(A) \, t^{n-j}. 
	\label{eq:det}
\end{equation}
We define the $k$-positive cone in $\Hm(n)$ by 
$$ \fG_k(\Hm(n)) = \{ A \in \Hm(n) \,\mid\, S_j(A)>0, j=1, \cdots, k \,\}.$$
Because of equation~(\ref{eq:det}), it is easy to see that  
$S_k$ is invariant under the adjoint action of $U(n)$, hence 
$\fG_k(\Hm(n))$ is also $U(n)$-invariant. Besides,  
all the the properties listed in Proposition~\ref{prop:2.1} are also true for
$S_k$ defined on $\Hm(n)$.  Especially, for any $A=(\abb{A}) \in
\Gamma_k(\Hm(n))$, the matrix with entries given by
\begin{equation}
  \label{eq:11}
	\uabb{F}= \frac{\di\log S_k(A)}{\di \abb{A}}
\end{equation}
is a  positive definite Hermitian matrix.

\bigskip
Now consider a complex vector space $V$ of complex dimension $n$ with
a fixed Hermitian metric $g$. Let $\fo$ be the Hermitian form of $g$. 
After fixing an unitary basis $\{\fc^1, \cdots, \fc^n\}$ for $V^*$,
any real $(1,1)$ form $\chi$ can be written as
$$ \chi = \ii \abb{\chi} \, \fc^{\fa} \wedge \overline{\fc^{\fb}},$$
where $A_{\chi}=\big(\abb{\chi})$ is a Hermitian matrix.
We define the $k$-th Hermitian $S_k(\chi)$ of $\chi$ with respect to $\fo$ as
$$ S_k(\chi) = S_k(A_{\chi}) = S_k\big( (\abb{\chi}) \big).$$
The definition of $S_k$ is independent of the choice of unitary basis,
in fact $S_k(\chi)$ can be defined without the use of unitary basis by
$$ \chi^k \wedge \fo^{n-k} = S_k(\chi)\,\fo^n. $$
Let $\fL_{\RR}^{1,1} V^*$ be the space of real $(1,1)$ form. We define
the $k$-positive cone in $\fL_{\RR}^{1,1}V^*$ by
$$ \fG_k(V^{1,1}) = \{\, \chi \in \fL_{\RR}^{1,1}V^* \,\mid\, S_j(\chi) > 0 , j =1,
\cdots, k \}.$$
All the 
properties listed in Proposition~\ref{prop:2.1} continue to be true 
for the $k$-th Hessian of real $(1,1)$ forms. 
Especially, for any $\chi_1, \cdots, \chi_k \in \fG_k(V^{1,1})$,
\begin{equation}
  \chi_1 \wedge \cdots \wedge \chi_k \wedge \fo^{n-k} > 0.
  \label{eq:12}
\end{equation}

\bigskip
Let $(M, \fo)$ be a \KH{} manifold with \KH{} form $\fo$.
The tangent space of $M$ at every point is a complex vector space
with Hermitian metric, so the construction of $\fG_k(V^{1,1})$ can be
carried out pointwise on $M$, hence we get a distribution of open
convex cones in the space of real $(1,1)$ form on $M$. 
Since the parallel transportation keeps the \KH{} metric, so this
distribution of convex cones is also invariant under the parallel 
transportation.  For simplicity, we still use $\fG_k$ to denote 
these convex cones.

\begin{defn}
  A real $(1,1)$ form $\chi \in \fO^{1,1}(M,\RR)$ is
  \emph{$k$-positive} with respect to $\fo$, if $\chi \in \fG_k$. 
\end{defn}

Let  $\cC^{\infty}(M,\RR)$ be the set of real valued smooth functions on $M$.
Denote
$$ \cP_k(M,\fo) = \{ \vff \in \cC^{\infty}(M,\RR) \,\mid\, \fo_\vff=\fo+ \ii \ddb \vff
\,\text{ is } \text{$k$-positive} \,\}. $$

\begin{prop}
  If $\vff \in \cC^{\infty}(M,\RR)$ solves (\ref{eq:1}), then $\vff
  \in \cP_k(M,\fo)$.
\end{prop}
\begin{proof}
  Since $\fo_{\vff}$ is a positive form at the point where $\vff$
  achieves minimum, so $\fo_{\vff} \in \fG_n \subset \fG_k$ at the
  minimum point. This together with the facts that $\fG_k \subset
  \Hm(n)$ is the connected component of $\{ S_k > 0 \}$ containing
  $\fG_n$, and the distribution of these cones is invariant under
  parallel transportation shows that $\fo_{\vff} \in \fG_k$ at every
  point, hence $\vff \in \cP_k(M,\fo)$.
\end{proof}

\begin{prop}
  If $\vff \in \cP_k(M,\fo)$, then the operator
  $$ \vff \mapsto S_k(\fo_{\vff})=\frac{\fo_\vff^k\wedge \fo^{n-k}}{\fo^n} $$
  is elliptic at $\vff$.
\end{prop}

\section{Uniqueness}
\label{sec:uniqueness}

Suppose both $\vff$ and $\psi$ solve equation (\ref{eq:1}), then
\begin{equation*}
  \begin{split}
    0 &= \int_M (\psi - \vff) \big( \fo_{\psi}^k\wedge \fo^{n-k} -
    \fo_{\vff}^k \wedge \fo^{n-k} \big) \\
      &= \sum_{l=0}^{k-1} \int_M (\psi - \vff)  \ii \ddb(\psi - \vff) \wedge
	  \fo_{\psi}^l\wedge\fo_{\vff}^{k-1-l} \wedge \fo^{n-k}  \\
      &= - \sum_{l=0}^{k-1} \int_M \ii \di(\psi - \vff) \wedge \dib(\psi - \vff) \wedge
	  \fo_{\psi}^l\wedge\fo_{\vff}^{k-1-l} \wedge \fo^{n-k}  \\
	  & \leq 0
  \end{split}
\end{equation*}
The last inequality is true because 
$$ \ii \di(\psi - \vff) \wedge \dib(\psi - \vff) \in \fG_n, \quad
\fo_{\vff} \in \fG_k \quad \text{and} \quad \fo_{\psi} \in \fG_k, $$ 
hence by equation~(\ref{eq:12}), for $l=0, \cdots, k-1$,
$$  \ii \di(\psi - \vff) \wedge \dib(\psi - \vff) \wedge
	  \fo_{\psi}^l\wedge\fo_{\vff}^{k-1-l} \wedge \fo^{n-k} \geq 0. $$  
So $\di (\psi - \vff)=0,$ and
$$ \fo_\vff = \fo_\psi. $$

\medskip
Same positivity argument can be used to prove
Proposition~\ref{prop:1.2} by Yau's Moser iteration, see
\cite{Siu1987} or \cite{Tian2000} for details.

\section{$C^1$ estimate}
\label{sec:C1}

We will follow B{\l}ocki's approach in \cite{Blocki2007} to get the
\emph{a priori} $C^1$ estimate. But unlike B{\l}ocki, we will use
the covariant derivative with respect to $\fo$ throughout this paper.

First let's fix some notation. Let $\{e_1, \cdots, e_n \}$ be a fixed
unitary frame for $(M,\fo)$, and $\{\fc^1, \cdots, \fc^n \}$ be the
duel frame. For $\vff \in \cP_k(M,\fo)$ with 
$$ \ii\ddb \vff = \ii \abb{\vff}\, \ua{\fc} \wedge \ubb{\fc}, $$
let
$$\uabb{G}=\uabb{F}\big(I_n+(\vff_{\fg\bar{\fd}})\big),$$
where $(\uabb{F})$ is the matrix-valued function defined on $\fG_k
(\Hm(n))$ by equation (\ref{eq:11}).
Let $(\abb{G})$ be the inverse matrix of $(\uabb{G})$, then 
\begin{equation}
	G=  \ii \abb{G}\,\ua{\fc}\wedge\ubb{\fc}
	\label{eq:G}
\end{equation}
induces a Hermitian metric $G$ on $M$, and this metric is independent
of the choice of unitary frame.

\bigskip
Suppose $\vff$ solves equation~(\ref{eq:1}) with
$\inf \vff =0$ and $\sup \vff= \osc \vff = C_0$. Let
$$ B=\norm{\grad \vff}^2_{\fo}= \vff_{,\fg}\vff^{\,,\fg} \quad \text{and}
\quad  A= \log B - h(\vff),$$
with $h(t)= \frac{1}{2} \log(2 t +1)$.
Let $\lap'=\uabb{G}\grad_{\overline{e_\fb}}\grad_{e_{\fa}}$, then
\begin{equation}
  \lap'A = \frac{\lap'B}{B} - \frac{1}{B^2} \uabb{G} B_{,\fa}B_{,\bfb} 
   - h'(\vff) \lap' \vff - h''(\vff) \uabb{G}\vff_{,\fa}\vff_{,\bfb}
\end{equation}
and
\begin{equation}
  \begin{split}
  \lap'B &=\uabb{G}\big(\vff^{\,,\fg}{}_{\fa\bfb} \vff_{,\fg} +
  \vff^{\,,\fg}{}_{\fa}\vff_{,\fg\bfb} +
  \vff^{\,,\fg}{}_{\bfb}\vff_{,\fg\fa} +
  \vff^{\,,\fg}\vff_{,\fg\fa\bfb}\big) \\
  & \geq \uabb{G}\big(\vff^{\,,\fg}{}_{\fa\bfb} \vff_{,\fg} +
       \vff^{\,,\fg}\vff_{,\fg\fa\bfb}\big).
  \label{eq:lapB}
\end{split}
\end{equation}
By Ricci identity
$$ \vff^{\,,\fg}{}_{\fa\bfb}= \vff_{,\fa\bfb}{}^{\fg} \quad \text{and}
\quad \vff_{,\fg\fa\bfb}=\vff_{,\fa\bfb\fg} + \vff_{,\eta}
R^{\,\eta}{}_{\fa\fg\bfb}.$$
Since $\vff$ solves equation~(\ref{eq:1}), so 
$$ \uabb{G} \vff_{,\fa\bfb\fg} = (\log f)_{,\fg} \quad
\text{and} \quad 
\uabb{G} \vff_{,\fa\bfb}{}^{\fg} = (\log f)^{\,,\fg}.$$
Hence
\begin{equation}
  \lap' B \geq 2 \Re\pair{\grad(\log f),\grad \vff}_{\fo} + \uabb{G}
  \vff_{,\eta}\vff^{\,,\fg} R^{\,\eta}{}_{\fa\fg\bfb}.
  \label{eq:15}
\end{equation}
If $(M,\fo)$ has non-negative holomorphic bisectional curvature, then
$$  \uabb{G} \vff_{,\eta}\vff^{\,,\fg} R^{\,\eta}{}_{\fa\fg\bfb} \geq 0 $$
therefore
\begin{equation}
  \lap' B \geq 2 \Re\pair{\grad(\log f),\grad \vff}_{\fo}.  \label{eq:16}
\end{equation}
Noticed that  $S_k$ is homogeneous of polynomial of degree $k$, so we have
$$ \lap'\vff = \uabb{G}\vff_{,\fa\bfb} = k - \tr_{\fo}G.$$
If $A$ achieves maximum at point $p$, then at point $p$, $\grad A =0$, i.e.
$$ \frac{B_{,\fa}}{B} = h'(\vff) \vff_{,\fa} \quad \text{and} \quad
\frac{B_{,\bfb}}{B}=h'(\vff)\vff_{,\bfb}.$$
Hence at the maximum point $p$,
$$ \lap' A \geq \frac{2}{B}\Re \pair{\grad(\log f),\grad \vff} -
(h''+h'^2) \uabb{G}\vff_{,\fa}\vff_{,\bfb} - kh' + h'\tr_{\fo}G. $$
Since $\vff$ takes value between $0$ and $C_0$, so
$$ \frac{1}{2C_0 + 1} \geq h' \geq 1 \quad \text{and} \quad 
 -h''-h'^2 \geq \frac{1}{(2C_0+1)^2}.$$
Therefore
\begin{equation}
  \lap' A \geq \frac{2}{B}\Re \pair{\grad(\log f),\grad \vff}+
  \frac{\uabb{G}\vff_{,\fa}\vff_{,\bfb}}{(2C_0+1)^2} - k + \frac{\tr_{\fo}G}{2C_0+1}.
  \label{eq:17}
\end{equation}
If the eigenvalue of $\fo+\ii\ddb\vff$ with respect to $\fo$ is
$(\fl_1,\cdots \fl_n)$, then
$$ \tr_{\fo} G = k \frac{S_{k-1}(\fl)}{S_k(\fl)} \quad \text{and}
\quad f=S_k(\fl). $$
Noticed
$$ |\grad \log f | =  \frac{k|\grad (f^{1/k})|}{f^{1/k}}
\leq |\grad (f^{1/k})| \tr_{\fo} G,$$
here we've used the Maclaurin Inequality
$$ \frac{1}{f^{1/k}} = \frac{1}{S_k^{1/k}(\fl)} \leq
\frac{S_{k-1}(\fl)}{S_k(\fl)} = \frac{1}{k}\tr_{\fo}G.$$
So at the maximum point $p$,
\begin{equation}
  0 \geq \big(\frac{1}{2C_0+1} - \frac{2|\grad(f^{1/k})|}{|\grad
  \vff|}\big) \tr_{\fo}G + \frac{1}{(2C_0+1)^2} \uabb{G} \vff_{,\fa}\vff_{,\bfb} - k
  \label{eq:18}
\end{equation}
We may also assume that 
$$\norm{\grad\vff}\leq  \frac{1}{2(2C_0+1)} \norm{\grad(f^{1/k})},$$
therefore at $p$, 
$$ k \geq  \frac{1}{(2C_0+1)^2} \uabb{G}\vff_{,\fa}\vff_{,\bfb} +
 \frac{1}{2(2C_0+1)} \tr_{\fo} G. $$ 
So at the maximum point $p$,
\begin{equation}
  \tr_{\fo} G \leq C, \quad \text{and} \quad \uabb{G}\vff_{,\fa}\vff_{,\bfb}
  \leq C 
  \label{eq:19}
\end{equation}
where $C$ is a constant depends on $C_0$ only.
By the Generalized Newton-Maclaurin Inequality,
\begin{equation}
  S_1(\fl) \leq S_k(\fl) \Big(\frac{S_{k-1}(\fl)}{S_k(\fl)}\Big)^{k-1}
  \leq f \Big(\frac{\tr_{\fo} G}{k} \Big)^{k-1},
  \label{eq:20}
\end{equation}
$S_1(\fl)$ is bounded by constant depends on $C_0$ and $\sup f$ at
point $p$. Noticed that $\forall\,\fl\in \RR^n$,
$$  \sum_j \fl_j^2 = (\sum_j \fl_j)^2 - \sum_{i\neq j} \fl_i\fl_j =
  (nS_1(\fl))^2 - n(n-1)S_2(\fl). $$
Since $\fl\in\fG_k$ with $k\geq 2$, so $S_2(\fl)\geq 0 $ and
\begin{equation}
 \sup_j |\fl_j| \leq n S_1(\fl). 
  \label{eq:21}
\end{equation}
Therefor at the maximum point $p$, the eigenvalue of $\fo_{\vff}$ with
respect to $\fo$ is bounded.  If we further assume that
$\fl_1\geq\cdots \geq \fl_n$, then the smallest eigenvalue of matrix
$(\uabb{G})$ is $\frac{1}{S_k(\fl)} \dod{S_k(\fl)}{\fl_1}$,
which can be bounded from below because of Lemma~\ref{lem:2.1},
\begin{equation}
  \frac{1}{S_k} \dod{S_k}{\fl_1} \geq  \frac{k}{n\fl_1}.
  \label{eq:22}
\end{equation}
Combine equation~(\ref{eq:19}) and (\ref{eq:20}), 
$\grad \vff$ is bounded at the maximum point of $A$, therefore $A$
is bounded everywhere, hence $\grad \vff$ is also bounded
everywhere, and this finishes the proof of Proposition~\ref{prop:1.3}.

\section{$C^2$ estimate}
\label{sec:C2}

Same as in Section~\ref{sec:C1}, we will use maximum principle to get
some \emph{a priori} estimate in this section. We can keep using the
Hermitian metric introduced by equation~(\ref{eq:G}). But in order to
get better regularity result, we will introduce a new Hermitian
metric. 
For $A=(\abb{A}) \in \fG_k(\Hm(n))$, denote 
$$ \uabb{\tilde{F}}(A) = \dod{S_k^{1/k}}{\abb{A}}, $$
then $\big(\uabb{\tilde{F}}(A)\big)$ is positive definite Hermitian
matrix. Using the same frame $(e_1,\cdots,e_n)$ and co-frame
$(\theta^1,\cdots, \theta^n)$ as in Section~\ref{sec:C1}, for $\vff\in \cP_k(M,\fo)$ with
$$ \ddb\vff = \vff_{,\fa\bfb} \theta^{\fa} \wedge
\overline{\theta^{\fb}}, $$ 
let
$$ \uabb{H} = \uabb{\tilde{F}}\big(I_n+(\vff_{,\fa\bfb})\big), $$
and $(\abb{H})= (\uabb{H})^{-1}$, then
\begin{equation}
  H= \ii \abb{H} \theta^{\fa}\wedge \overline{\theta^{\bf}} 
  \label{eq:H}
\end{equation}
induces a Hermitian metric on $M$.

Let $\lapx = \uabb{H} \nabla_{\overline{e_{\fb}}} \nabla_{e_{\fa}}$, then
$$\lapx(n+\lap \vff) =\uabb{H} \vff_{,\fg\bfg\fa\bfb}.$$
From the equation
$$  S^{1/k}_k\big(I_n + (\abb{\vff}) \big) = f^{1/k}, $$
one get
$$ \uabb{H} \vff_{,\fa\bfb\fg\bfg} \geq \lap (f^{1/k}).$$
Here we've used the concavity of $S_k^{1/k}$.  By Ricci identity, 
\begin{equation}
  \vff_{,\fa\bfb\fg\bfg} = \vff_{,\fg\bfg\fa\bfb} +
  \vff_{,\xi\bfb}R^{\xi}{}_{\fa\fg\bfg} -
  \vff_{,\xi\bfg}R^{\xi}{}_{\fg\fa\bfb}
\end{equation}
So
\begin{equation}
	\lapx(n+\lap \vff) \geq \lap f + \uabb{H} 
	\vff_{,\xi\bfg}R^{\xi}{}_{\fg\fa\bfb} - \uabb{H} 
	\vff_{,\xi\bfb}R^{\xi}{}_{\fa\fg\bfg} 
\end{equation}
Choose a unitary frame so that $\abb{\vff}$ is diagonal matrix, then
\begin{equation*}
  \begin{split}
	& \uabb{H} \vff_{,\xi\bfg}R^{\xi}{}_{\fg\fa\bfb} - \uabb{H}
	  \vff_{,\xi\bfb}R^{\xi}{}_{\fa\fg\bfg}   \\
    =& \sum_{\fa,\fg} (H^{\fa\bfa} \vff_{,\fg\bfg}
	  R{\bfg\fg\fa\bfa} - H^{\fa\bfa}\vff_{,\fa\bfa}
	  R_{\bfa\fa\fg\bfg}) \\
	=& \sum_{\fa,\fg} H^{\fa\bfa} R_{\bfg\fg\fa\bfa} (\vff_{,\fg\bfg}
	  - \vff_{,\fa\bfa}) \\
	=& \sum_{\fa < \fg} R_{\bfg\fg\fa\bfa} (\vff_{,\fg\bfg}
	- \vff_{,\fa\bfa}) (H^{\fa\bfa} - H^{\fg\bfg})
  \end{split}
\end{equation*}
If $\fo_{\vff}$ is diagonalized as $(\fl_1, \cdots, \fl_n)$ with respect
to $\fo$, then 
$$ \vff_{\fa\bfa} = \fl_{\fa} - 1 \quad \text{and} \quad
H^{\fa\bfa} = \frac{1}{k} S_k^{\frac{1}{k}-1} \dod{S_k}{\fl_{\fa}}. $$
Hence 
$$\vff_{,\fa\bfa} \geq \vff_{,\fg\bfg} \quad \Rightarrow  \quad
\fl_{\fa} \geq \fl_{\fg} \quad \Rightarrow \quad H^{\fa\bfa} \leq H^{\fg\bfg},$$
therefore
$$ (\vff_{,\fg\bfg} - \vff_{,\fa\bfa}) (H^{\fa\bfa}-H^{\fg\bfg}) \geq 0. $$
If $(M,\fo)$ has non-negative holomorphic bisectional curvature, i.e.
$$ R_{\bfg\fg\fa\bfa} \geq 0, $$
then
$$ 
	\uabb{H} \vff_{,\xi\bfg}R^{\xi}{}_{\fg\fa\bfb} - \uabb{H}
	  \vff_{,\xi\bfb}R^{\xi}{}_{\fa\fg\bfg} 
	\geq 0 
$$
and
\begin{equation}
  \lapx(n+\lap \vff) \geq \lap(f^{1/k}).
  \label{eq:27}
\end{equation}

Consider $n+\lap \vff - \vff$. At the point where
$n+\lap\vff+\vff$ achieves the maximum, 
$$  0 \geq \lapx (n+\lap\vff-\vff) \geq \lap(f^{1/k}) - f^{1/k} +
\tr_{\fo} H. $$
Hence at the maximum point
$$ \tr_{\fo} H \leq f^{1/k} - \lap(f^{1/k}) \leq C $$
where $C$ is a constant depends on $\sup f$ and $\inf \lap(f^{1/k})$.
Noticed that
$$ \tr_{\fo} H = k S_k^{\frac{1}{k}} \frac{S_{k-1}}{S_k}, $$
and
$$ n+\lap \vff =  S_1(\fl) \leq S_k\Big(\frac{S_{k-1}}{S_k}\Big)^{k-1} 
= S_k^{\frac{1}{k}} \Big(S_k^{\frac{1}{k}} \frac{S_{k-1}}{S_k} \Big)^{k-1}, $$
so by bounding $\tr_{\fo} H$, we also bound
$n+\lap \vff$ by constant depends on $\sup f$ and $\lap(f^{1/k})$ at
the maximum point of $n+\lap \vff - \vff$. Therefore get a global
bound for $n+\lap \vff$. Then by equation~(\ref{eq:21}), we also get
global bound for the eigenvalues of $\fo_{\vff}$ with respect to
$\fo$, i.e. 
$$ \norm{\ddb\vff}_{C^0(\fo)} \leq C_2 $$
with $C_2$ depends on $\sup f$, $\inf \lap(f^{1/k})$ and $\osc \vff$.


\bibliographystyle{alpha}
\bibliography{kpluri}

\end{document}